\documentclass[reqno, 11pt]{amsart}
\usepackage[top=4cm, bottom=3cm, left=3.2cm, right=3.2cm]{geometry}
\usepackage{enumerate}
\usepackage{amssymb,amsmath}

\newtheorem{thm}{Theorem}

\newtheorem{lem}{Lemma}

\theoremstyle{definition}

\theoremstyle{remark}
\newtheorem{rem}{Remark}

\allowdisplaybreaks

\newcommand{\dint}{\displaystyle\int}

\newcommand{\R}{\mathbb{R}}

\numberwithin{equation}{section} \numberwithin{lem}{section}
\numberwithin{thm}{section} \numberwithin{prop}{section}
\numberwithin{cor}{section} \numberwithin{rem}{section}
\title[Global existence under sharp condition]{Supercritical degenerate parabolic-parabolic Keller-Segel system --
existence criterion given by the best constant in Sobolev's inequality}

\author{Jinhuan Wang$^{\,1}$, Yue Li$^{\,1}$ and Li Chen$^{\,2}$}
\thanks{The work of J. Wang is partially supported by Program for Liaoning Excellent Talents in University (Grant No. LJQ2015041) and Key Project of Education Department of Liaoning Province (Grant No. LZD201701).}
\thanks{Corresponding author: Li Chen. She is partially supported by the DAAD project ``DAAD-PPP VR China" (Project-ID: 57215936).}
\begin{document}
\maketitle
\begin{center}
{\footnotesize
1-School of Mathematics, Liaoning University, Shenyang, 110036, P. R. China \\
 email: wjh800415@163.com\\
 \smallskip
2-Lehrstuhl f\"ur Mathematik IV, Universit\"at Mannheim, 68131, Mannheim.
email: chen@math.uni-mannheim.de
}
\end{center}
\maketitle
\date{}
\begin{abstract}
This article presents a relationship between the sharp constant of the Sobolev inequality and the initial criterion to the global existence of degenerate parabolic-parabolic Keller-Segel system with the diffusion exponent $\frac{2n}{2+n}<m<2-\frac{2}{n}$.
The global weak solution obtained in this article does not need any smallness assumption on the initial data. Furthermore, a uniform in time $L^{\infty}$ estimate of the weak solutions is obtained via the Moser iteration, where the constant in $L^p$ estimate for the gradient of the chemical concentration has been exactly formulated in order to complete the iteration process.
\end{abstract}


\newtheorem{theorem}{Theorem}[section]
\newtheorem{definition}{Definition}[section]
\newtheorem{lemma}{Lemma}[section]
\newtheorem{proposition}{Proposition}[section]
\newtheorem{corollary}{Corollary}[section]
\newtheorem{remark}{Remark}[section]
\renewcommand{\theequation}{\thesection.\arabic{equation}}
\catcode`@=11 \@addtoreset{equation}{section} \catcode`@=12


\section{Introduction}
In this paper, we study the following degenerate parabolic-parabolic Keller-Segel equations with diffusion exponent $\frac{2n}{2+n}<m<2-\frac{2}{n}$ in dimension $n\geq 3$,
\begin{eqnarray}
&&\rho_t=\Delta \rho^m -{\rm div}(\rho\nabla c),\quad x\in \mathbb R^n ,t\geq 0,\label{equ1}\\
&&c_t=\Delta c- c+\rho, \quad x\in \mathbb R^n ,t\geq 0,\label{equ2}\\
&&\rho(x,0)=\rho_{0}(x), ~~c(x,0)=c_{0}(x),\quad x\in \mathbb R^n,\label{initialdata}
\end{eqnarray}
where $\rho(x,t)$ and $c(x,t) $ represent respectively the density of bacteria and the chemical substance concentration. One important feature of this system is the conservation of mass
\begin{eqnarray}\label{masscons}
\dint_{\R^n}\rho(x,t)dx=\dint_{\R^n}\rho(x,0)dx=:M_0.
\end{eqnarray}

The system (\ref{equ1})-(\ref{equ2}) can be  recast into  the following gradient flow structures
\begin{eqnarray}\label{gradientflow}
\rho_{t}=\nabla\cdot\left(\rho \nabla \frac{\delta \mathcal{F}(\rho,c)}{\delta \rho}\right),\quad c_{t}=-\frac{\delta \mathcal{F}(\rho,c)}{\delta c},
\end{eqnarray}
where
\begin{eqnarray*}
\frac{\delta \mathcal{F}(\rho,c)}{\delta \rho}=\frac{m}{m-1}\rho^{m-1}-c\quad \mbox{ and }\quad \frac{\delta \mathcal{F}(\rho,c)}{\delta c}=-\Delta c+c-\rho
\end{eqnarray*}
are respectively the first order variational of the following energy functional with respect to functions $\rho$ and $c$,
\begin{eqnarray}\label{entropy}
\mathcal{F}(\rho,c)= \int_{\R^n} \big(\frac{1}{m-1}\rho^m -\rho c +\frac{1}{2}|\nabla c|^2+\frac{1}{2}c^2\big)dx.
\end{eqnarray}
From (\ref{gradientflow}), we easily know that the following energy-dissipation relation holds
\begin{eqnarray*}\label{entropyineq}
\frac{d}{dt}\mathcal{F}(\rho(\cdot,t),c(\cdot,t)) +\int_{\mathbb{R}^n} \Big(\rho \big|\nabla(\frac{m}{m-1}\rho^{m-1}-c)\big|^2+|c_t|^2\Big) dx= 0.
\end{eqnarray*}

The system (\ref{equ1})-(\ref{initialdata}) with different $m$ has been widely studied since the last decade. Sugiyama et al. in \cite{S06,SK06} and Ishida et al. in \cite{IY12} proved that if $m>2-\frac{2}{n}$, the solution exists globally for any initial data; if $1<m\leq 2-\frac{2}{n}$, for the parabolic-elliptic case both global existence and blow-up can happen for specifically selected initial data \cite{SK06}; for the  parabolic-parabolic case global existence was proved for small initial data \cite{IY12}. The diffusion exponent $m^*:=2-\frac{2}{n}$ comes from the scaling invariance of the total mass, which is sometimes called Fujita type exponent.

 Since 2009, the sharp initial criterion of degenerate parabolic-elliptic Keller-Segel equations has attracted more and more attentions. Blanchet et al. in \cite{BCL09} studied the degenerate case with diffusion exponent $m=m^*$, where a critical mass was given. Later on, Chen et al. \cite{CLW1} found another critical diffusion exponent $m_c:=\frac{2n}{n+2}$, which comes from the conformal invariance of the free energy. In \cite{CLW1} it was shown that the $L^{m_c}$ norm of positive stationary solutions can be viewed as the criterion for the global existence and blow-up of solutions. In the case $m_c < m < m^*$ (see \cite{CW}),
there is a constant $s^*$ depending only on the initial mass and the best constant of the Hardy-Littlewood-Sobolev inequality to classify the initial data in order to achieve global existence or blow-up. More precisely, for $\rho_0\in L^1(\mathbb{R}^n)\cap L^\infty(\mathbb{R}^n)$
there is a {\em unique} global weak solution if $\|\rho_0\|_{L^{m_c}(\R^n)}$ is less than $s^*$,
and blow-up occurs in finite time if $\|\rho_0\|_{L^{m_c}(\R^n)}$ is larger than $s^*$ (see \cite{CW,LW}).

However, the best criterion to the parabolic-parabolic Keller-Segel system has not been well developed comparing to its parabolic-elliptic version, due to the fact that the symmetric structure in the aggregation is missing.
In the paper \cite{B13}, Blanchet et al. showed the global existence for the case $m=m^*$ when the chemotactic sensitivity is smaller than a threshold, which is exactly the same as that in the parabolic-elliptic case.
One of the purposes of this paper is to derive the condition $\|\rho_0\|_{L^{m_c}(\R^n)}<s^*$ for global existence when $m_c<m<m^*$. Here $s^*$ is also exactly the same as that of the parabolic-elliptic case in \cite{CW}.

The Sobolev inequality with the best constant (for example in \cite{LiebAnalysis})
\begin{eqnarray}\label{Sobolev}
S_n \|u\|_{L^{\frac{2n}{n-2}}(\R^n)}^2\leq \|\nabla u\|_{L^2(\R^n)}^2, \quad  S_n=\dfrac{n(n-2)}{4}2^{\frac{2}{n}}\pi^{1+\frac{1}{n}}\Gamma\Big(\frac{n+1}{2}\Big)^{-\frac{2}{n}}
\end{eqnarray}
plays an important role in the decomposition of the free energy (\ref{entropy}). More precisely, the free energy has the following lower bound for $m>\frac{2n}{n+2}$, i.e.
\begin{align*}
F(\rho)= &\dint_{\R^n} \Big(\frac{1}{m-1}\rho^m -\rho c +\dfrac{1}{2}|\nabla c|^2+ |c|^2\Big)dx\nonumber\\
\nonumber \geq & \dfrac{1}{m-1}\|\rho\|_{L^m(\R^n)}^m -\dfrac{1}{2S_n} \|\rho\|_{L^{\frac{2n}{n+2}}(\R^n)}^2 -\dfrac{S_n}{2}\|c\|_{L^{\frac{2n}{n-2}}(\R^n)}^2+ \frac{1}{2}\|\nabla c\|_{L^2(\R^n)}^2\\
\geq& \Big(\dfrac{1}{m-1} M_0^{\frac{2n-m(n+2)}{n-2}}-\dfrac{1}{2S_n} \|\rho\|_{L^{\frac{2n}{n+2}}(\R^n)}^{2-\frac{2n(m-1)}{n-2}}\Big) \|\rho\|_{L^{\frac{2n}{n+2}}(\R^n)}^{\frac{2n(m-1)}{n-2}} .
\end{align*}
Then after appropriate assumption on initial free energy, the estimate on $\|\rho\|_{L^{\frac{2n}{n+2}}(\R^n)}$ can be obtained. It is the milestone in getting global existence of weak solutions. Actually, there is a big class of evolutionary equations that have close connections with corresponding Sobolev inequalities\cite{LW1,W}. We refer to \cite{DELL} given by Dolbeault et al. for recent developments.

The main result of this paper is given by the following theorem
\begin{thm}\label{thm}
Assume that the initial density $\rho_0\in L^1_+(\mathbb R^{n})\cap L^{m} (\mathbb R^{n})$ and ${\mathcal F}(\rho_0)<{\mathcal F}^*$,  $\|\rho_0\|_{L^{\frac{2n}{n+2}} (\mathbb R^{n})}< (s^*)^{\frac{n-2}{2n(m-1)}}$ and $\nabla c_0\in L^{\infty}(\R^n)$, then (\ref{equ1})-(\ref{initialdata}) has a global weak solution, i.e. for all $T>0$ and some $1<r,s\leq 2$, there is a function $\rho(x,t)$ with
\begin{align}
&\rho\in L^\infty(0,+\infty; L^1_+(\mathbb R^{n})\cap L^{m}(\mathbb R^n)),\label{r1}\\
& \nabla \rho \in L^2(0,T;L^r(\mathbb R^n)), ~~\partial_t\rho\in L^2(0,T;(W_{\rm loc}^{1, s'}(\mathbb{R}^n))'),\quad \frac{1}{s'}+\frac{1}{s}=1,\label{r2}
\end{align}
such that it satisfies (\ref{equ1})-(\ref{initialdata}) in the sense of distribution.
Here ${\mathcal F}^*$ and $s^*$ are universal constants given by
\begin{align*}
{\mathcal F}^*=&\frac{2n-2-mn}{(m-1)(n-2)}M_0^{\frac{2n-mn-2n}{2n-2-mn}}\Big(\frac{2nS_n}{n-2}\Big)^{\frac{nm-n}{2n-2-mn}}>0,\\ s^*=&\left(\frac{2nS_n}{n-2} M_0^{\frac{2n-m(n+2)}{n-2}}\right)^{\frac{n(m-1)}{2n-2-mn}}>0,
\end{align*}
where $M_0$ is the initial mass $\|\rho_0\|_{L^1(\mathbb{R}^n)}$ defined in (\ref{masscons}) and $S_n$ is the best constant of the Sobolev inequality, see (\ref{Sobolev}).
\end{thm}
\begin{rem}
For parabolic-parabolic Keller-Segel system, the initial criterion $s^*$ is exactly the same as that of the parabolic-elliptic case \cite{CW}. In fact, Sobolev's inequality in parabolic-parabolic system plays the same role as that of the critical Hardy-Littlewood-Sobolev (H.-L.-S.) inequality in the parabolic-elliptic case. In order to compare, we cite H.-L.-S. inequality from  \cite{LiebAnalysis}.
For $\rho\in L^{\frac{2n}{n+2}}({\mathbb R^n})$, it holds that
\begin{eqnarray*}
\dint\dint_{\mathbb
R^n\times \mathbb
R^n}\dfrac{\rho(x)\rho(y)}{|x-y|^{n-2}} dx dy \leq C(n)
\|\rho\|_{L^{\frac{2n}{n+2}}(\R^n)}^2,
\end{eqnarray*}
where the best constant
\begin{eqnarray*}
C(n)=\pi^{(n-2)/2}\dfrac{1}{\Gamma(n/2+1)}\left\{\dfrac{\Gamma(n/2)}{\Gamma(n)}\right\}^{-2/n}.
\end{eqnarray*}
This has been used to derive the initial criterion in \cite{CW}
$$
s^*=\Big(\dfrac{2n^2\alpha(n)M_0^{\frac{2n-m(n+2)}{n-2}}}{C(n)}\Big)^{\frac{n(m-1)}{2n-2-mn}}.
$$
\end{rem}

\vskip 5mm
 Furthermore, we get the following uniform in time $L^{\infty}$-bound of weak solutions.
\begin{thm}\label{unithm}
Assume that the initial density $\rho_0\in L^1_+(\mathbb R^{n})\cap L^{\infty} (\mathbb R^{n})$ and $\nabla c_0\in L^{\infty}(\mathbb{R}^n)$. Let $(\rho,c)$ be a weak solution to the system (\ref{equ1})-(\ref{initialdata}). Then there is a constant $C$ only depends on $M_0,\|\rho_0\|_{L^{\infty}(\mathbb{R}^n)}$ and $\|\nabla c_0\|_{L^{\infty}(\mathbb{R}^n)}$ such that
\begin{eqnarray}\label{1.7}
\|\rho\|_{L^\infty(0,\infty; L^{\infty}(\mathbb{R}^{n}))}\leq C.
\end{eqnarray}
\end{thm}

Next we briefly explain the main idea and the key points in getting $L^{\infty}$ estimate by using Moser iteration. The iteration is based on the following $L^p$ estimate
\begin{align}
\frac{d}{dt}\dint_{\mathbb{R}^n} \rho^{p} \,dx
&=-mp(p-1)\dint_{\mathbb{R}^n}\rho^{m+p-3}|\nabla\rho|^2 \,dx +p(p-1)\dint_{\mathbb{R}^n}\rho^{p-1}\nabla c\cdot\nabla\rho\,dx\label{na}\\
\mbox{or }\quad&=-mp(p-1)\dint_{\mathbb{R}^n}\rho^{m+p-3}|\nabla\rho|^2 \,dx -(p-1)\dint_{\mathbb{R}^n}\rho^{p}\Delta c\,dx.\label{de}
\end{align}
 From (\ref{na}) and (\ref{de}), we know that it is necessary to estimate $\Delta c$ or $\nabla c$ in terms of $\rho$. By a special case of famous maximal $L^p$-regularity theorem,
which was proved by Matthias and Jan in \cite{MJ}, the estimate of $\Delta c$ is given by
\begin{eqnarray} \label{delta}
\|\Delta c\|_{L^p(0, T;L^p(\mathbb{R}^n))}\leq C_p(\|\rho\|_{L^p(0, T;L^p(\mathbb{R}^n))}+C),
\end{eqnarray}
where $C_p$ is a constant dependent of $p$ and the constant $C$ depends on initial data $c_0$. However, the exact formulation of $C_p$ was not written down in \cite{MJ}. A direct application of (\ref{delta}) without knowing $C_p$ explicitly will bring troubles in the iteration for $p\gg 1$.
Instead, we give the estimate for $\|\nabla c\|_{L^{\infty}(\mathbb{R}^n)}$ with an explicit constant by using semigroup representation of (\ref{equ2}), the properties of the heat kernel and Young's inequality for the convolution. Although the following discussion is standard, we have not found an appropriate reference to cite. Therefore, we do it by hand.

The mild solution to the Cauchy problem
\begin{eqnarray*}
&&\partial_{t}c=\Delta c-c+f,\quad x\in \mathbb R^n ,t\geq 0,\label{cregequ2}\\
&&c(x,0)=c_{0}(x), \quad x\in \mathbb R^n\label{creginitial}
\end{eqnarray*}
is given by
$$
c(x,t)=e^{-t}e^{t\Delta}c_0+\int_0^t e^{s-t}e^{(t-s)\Delta}f(x,s)\,ds,
$$
where the heat semigroup operator
$e^{t\Delta}$ is defined by
$$
(e^{t\Delta}h)(x, t) := G(x, t) \ast h(x, t),\quad G(x, t) = \frac{1}{(4\pi t)^{\frac{n}{2}}}e^{-\frac{|x|^2}{4t}}.
$$
From the classical semigroup theory, we know that if $c_0 \in L^q(\mathbb{R}^n)$ and $f\in L^1(0, T; L^q(\mathbb{R}^n))$, where $T > 0$, $1\leq q \leq \infty$, then
 $c(x, t) \in C([0, T]; L^q(\mathbb{R}^n))$.

It follows immediately from the Young inequality for the convolution \cite[pp. 99]{LiebAnalysis} that for any $1\leq q\leq p\leq +\infty$ and all $t>0$, $f\in L^{q}(\mathbb{R}^n)$, it holds that
\begin{eqnarray}
&&\|e^{t\Delta} f(\cdot,t)\|_{L^p(\R^n)}\leq A_{p,q,n}t^{-\frac{n}{2}(\frac{1}{q}-\frac{1}{p})}\|f(\cdot,t)\|_{L^q(\R^n)},\nonumber\\
&&\|\nabla e^{t\Delta} f(\cdot,t)\|_{L^p(\R^n)}\leq B_{p,q,n}t^{-\frac{1}{2}-\frac{n}{2}(\frac{1}{q}-\frac{1}{p})}\|f(\cdot,t)\|_{L^q(\R^n)}.\label{nabla heatk1}
\end{eqnarray}
Here
\begin{eqnarray*}
A_{p,q,n}=\frac{C_qC_r}{C_p}\frac{(2^{n-1}S_n\Gamma(\frac{n}{2}))^{\frac{1}{r}}}{(4\pi)^{\frac{n}{2}}r^\frac{n}{2r}},\quad
B_{p,q,n}=\frac{C_qC_r}{C_p}\frac{(2^{n-1}S_n
\Gamma(\frac{r}{2}+\frac{n}{2}))^{\frac{1}{r}}}{(4\pi)^{\frac{n}{2}}r^{\frac{1}{2}+\frac{n}{2r}}},\label{aandB}
\end{eqnarray*}
where $r$ satisfies $\frac{1}{r}=1+\frac{1}{p}-\frac{1}{q}$ and $C_q=q^{\frac{1}{q}-\frac{1}{2}}(q-1)^{\frac{1}{2}-\frac{1}{2q}}$ defined in \cite[pp. 98]{LiebAnalysis}.

Assuming $c_0\in W^{1,p}(\mathbb{R}^n)$ and $f\in L^{\infty}(0,\infty;L^q(\mathbb{R}^n))$, where $1 \leq q \leq p \leq \infty$, $\frac{1}{q}-\frac{1}{p}<\frac{1}{n}$. A direct application of the inequalities in (\ref{nabla heatk1}) and the Bochner Theorem in \cite[pp.650]{E}, we have for $t\in [0,\infty)$
\begin{align}
&\|c(\cdot,t)\|_{L^p(\R^n)}\leq \|c_0\|_{L^{p}(\R^n)}+A_{p,q,n}~\Gamma \left(1-\big(\frac{1}{q}-\frac{1}{p}\big)\frac{n}{2}\right)\|f\|_{L^{\infty}(0,\infty;~L^q(\mathbb{R}^n))},\label{cLp}\\
&\|\nabla c(\cdot,t)\|_{L^p(\R^n)}\leq \|\nabla c_0\|_{L^{p}(\R^n)}+B_{p,q,n} ~\Gamma \left(\frac{1}{2}-\big(\frac{1}{q}-\frac{1}{p}\big)\frac{n}{2}\right)\|f\|_{L^{\infty}(0,\infty;~L^q(\mathbb{R}^n))}.\label{nablacLp}
\end{align}
In particular, for $p=\infty$, $q>n$ we have
\begin{eqnarray}
\|\nabla c(\cdot,t)\|_{L^{\infty}(\R^n)}\leq \|\nabla c_0\|_{L^{\infty}(\R^n)}+B_{\infty,q,n} ~\Gamma \left(\frac{1}{2}-\frac{n}{2q}\right)\|f\|_{L^{\infty}(0,\infty;~L^q(\mathbb{R}^n))}.\label{nablacLinfty}
\end{eqnarray}
This key estimate with explicit constant $B_{\infty,q,n}$ is going to be used in (\ref{na}) in order to finish the iteration.
\vskip 5mm

This paper is arranged as follows.
The global existence of weak solutions for the system (\ref{equ1})-(\ref{initialdata}) is proved in Section 2.
Furthermore, a uniform in time $L^{\infty}$ bound of weak solutions is given in Section 3.

\section{Regularized problem and global existence of weak solutions}
In this section, for completeness, in order to prove the global existence , we start from the regularized problem of the system (\ref{equ1})-(\ref{initialdata}) and prove uniform estimates. The regularized problem is
\begin{eqnarray}
&\partial_{t}\rho_{\varepsilon}=\Delta(\rho_{\varepsilon}+\varepsilon)^{m}-{\rm div}((\rho_{\varepsilon}+\varepsilon)\nabla (c_{\varepsilon}\ast J_{\varepsilon})),& x\in \mathbb R^n ,t\geq 0,\label{regequ1}\\
&\partial_{t}c_{\varepsilon}=\Delta c_{\varepsilon}-c_{\varepsilon}+\rho_{\varepsilon}\ast J_{\varepsilon},& x\in \mathbb R^n ,t\geq 0,\label{regequ2}\\
&\rho(x,0)=\rho_{0,\varepsilon}(x), c(x,0)=c_{0,\varepsilon}(x), & x\in \mathbb R^n,\label{reginitial}
\end{eqnarray}
where $J_{\varepsilon}$ is a mollifier with radius $0<\varepsilon\ll 1$ satisfying  $\dint_{\R^n} J_{\varepsilon}dx=1$. And the initial data are also mollified, i.e. $\rho_{0,\varepsilon}=\rho_{0}\ast J_{\varepsilon}$ and $c_{0,\varepsilon}=c_{0}\ast J_{\varepsilon}$. Under the assumptions on $\rho_0,c_0$, it is obvious that
$\rho_{0,\varepsilon}(x)\geq0$, $c_{0,\varepsilon}(x)\geq0$ for any $x\in \mathbb R^n$,
$\rho_{0,\varepsilon}\in L^r(\mathbb{R}^n)$ for $r\geq 1$, $\nabla c_{0,\varepsilon}\in L^{\infty}(\R^n)$, and $\|\rho_{0,\varepsilon}\|_{L^1(\R^n)}=\|\rho_0\|_{L^1(\R^n)}=M_0$.

The classical parabolic theory implies that the above regularized problem (\ref{regequ1})-(\ref{reginitial}) has a global smooth non-negative solution $\rho_\varepsilon$ if the initial data is non-negative. Notice that for the solution of (\ref{regequ1})-(\ref{reginitial}), it holds that $\int_{\mathbb R^n}\rho_{\varepsilon}(x,t)\,dx=M_0$.

Throughout this section, we denote by $C(m,n,p)$ a constant only depends on $m,n,p$, $M_0$, $\|\rho_0\|_{L^{\infty}(\R^n)}$ and $\|\nabla c_0\|_{L^{\infty}(\R^n)}$, which may be different from line to line.

Now we focus on uniform estimates of solutions to (\ref{regequ1})-(\ref{reginitial}).
\subsection{Variational structure of the regularized problem and initial criterion} In this subsection we show the variational structure of the regularized problem and use it to deduce a uniform bound on the $L^{\frac{2n}{n+2}}(\mathbb R^n)$ norm of solutions to the regularized problem.

The regularized problem (\ref{regequ1})-(\ref{reginitial}) can be written as the following variational form
\begin{eqnarray}\label{reggradientflow}
\partial_{t}\rho_{\varepsilon}=\nabla\cdot\left((\rho_{\varepsilon}+\varepsilon) \nabla \frac{\delta \mathcal{F}_{\varepsilon}(\rho_{\varepsilon},c_{\varepsilon})}{\delta \rho_{\varepsilon}}\right),\quad \partial_{t} c_{\varepsilon}=-\frac{\delta \mathcal{F}_{\varepsilon}(\rho_{\varepsilon},c_{\varepsilon})}{\delta c_{\varepsilon}},
\end{eqnarray}
where
\begin{eqnarray*}
&&\frac{\delta \mathcal{F}_{\varepsilon}(\rho_{\varepsilon},c_{\varepsilon})}{\delta \rho_{\varepsilon}}=\frac{m}{m-1}(\rho_{\varepsilon}+\varepsilon)^{m-1}-c_{\varepsilon}\ast J_{\varepsilon},\\
&&\frac{\delta \mathcal{F}_{\varepsilon}(\rho_{\varepsilon},c_{\varepsilon})}{\delta c_{\varepsilon}}=-\Delta c_{\varepsilon}+c_{\varepsilon}-\rho_{\varepsilon}\ast J_{\varepsilon}
\end{eqnarray*}
 are respectively the first order variational of the following energy functional respect to functions $\rho_{\varepsilon}$ and $c_{\varepsilon}$,
\begin{align}\label{regentropy}
\mathcal F_{\varepsilon}(\rho_{\varepsilon}, c_{\varepsilon})=&\frac{1}{m-1}\dint_{\R^n}\big((\rho_{\varepsilon}+\varepsilon)^{m}-\varepsilon^{m})dx-\dint_{\R^n}\big(\rho_{\varepsilon}\ast J_{\varepsilon}\big) c_{\varepsilon}\,dx\nonumber\\
&+\frac{1}{2}\dint_{\R^n}|\nabla c_{\varepsilon}|^{2}dx+\frac{1}{2}\dint_{\R^n}c_{\varepsilon}^{2}\,dx.
\end{align}
From (\ref{reggradientflow}), we know that the following energy-dissipation relation holds
\begin{eqnarray*}
\dfrac{d}{dt}\mathcal F_{\varepsilon}(\rho_{\varepsilon},c_{\varepsilon}) +\dint_{\R^n} \left((\rho_{\varepsilon}+\varepsilon) \Big|\nabla\big(\dfrac{m}{m-1}(\rho_{\varepsilon}+\varepsilon)^{m-1}-c_{\varepsilon}\ast J_{\varepsilon}\big)\Big|^2+|\partial_{t}c_\varepsilon  |^2\right) dx= 0.
\end{eqnarray*}
The monotone decreasing property of the free energy $\mathcal F_{\varepsilon}(\rho_{\varepsilon}, c_{\varepsilon})$ follows immediately by the non-negativity of the entropy production.

Next, our purpose is to decompose the free energy defined in (\ref{regentropy}) into two parts by using the Sobolev inequality (\ref{Sobolev}) and the H\"older inequality.

As a preparation, we have
\begin{align*}
\dint_{\R^n}\big(\rho_{\varepsilon}\ast J_{\varepsilon}\big)  c_{\varepsilon} \,dx \leq& \|\rho_{\varepsilon}\ast J_{\varepsilon}\|_{L^{\frac{2n}{n+2}}(\R^n)} \|c_{\varepsilon}\|_{L^{\frac{2n}{n-2}}(\R^n)}\nonumber\\
\leq&\|\rho_{\varepsilon}\|_{L^{\frac{2n}{n+2}}(\R^n)} \|c_{\varepsilon}\|_{L^{\frac{2n}{n-2}}(\R^n)}\nonumber\\
 \leq& \dfrac{1}{2S_n} \|\rho_{\varepsilon}\|_{L^{\frac{2n}{n+2}}(\R^n)}^2 +\dfrac{S_n}{2}\|c_{\varepsilon}\|_{L^{\frac{2n}{n-2}}(\R^n)}^2.
\end{align*}
Therefore, we decompose the free energy into the following form
\begin{align*}
\mathcal F_{\varepsilon}(\rho_{\varepsilon},c_{\varepsilon})= &\dint_{\R^n} \Big(\frac{1}{m-1}((\rho_{\varepsilon}+\varepsilon)^m-\varepsilon^{m}) -\big(\rho_{\varepsilon}\ast J_{\varepsilon}\big) c_{\varepsilon} +\dfrac{1}{2}(|\nabla c_{\varepsilon}|^2+ |c_{\varepsilon}|^2)\Big)dx\nonumber\\
 \geq & \dfrac{1}{m-1}\|\rho_{\varepsilon}\|_{L^{m}(\R^n)}^m -\dfrac{1}{2S_n} \|\rho_{\varepsilon}\|_{L^{\frac{2n}{n+2}}(\R^n)}^2 -\dfrac{S_n}{2}\|c_{\varepsilon}\|_{L^{\frac{2n}{n-2}}(\R^n)}^2+ \frac{1}{2}\|\nabla c_{\varepsilon}\|_{L^{2}(\R^n)}^2\nonumber\\
 =& \left(\dfrac{1}{m-1}\|\rho_{\varepsilon}\|_{L^{m}(\R^n)}^m -\dfrac{1}{2S_n} \|\rho_{\varepsilon}\|_{L^{\frac{2n}{n+2}}(\R^n)}^2\right)+ \left(-\dfrac{S_n}{2}\|c_{\varepsilon}\|_{L^{\frac{2n}{n-2}}(\R^n)}^2+ \frac{1}{2}\|\nabla c_{\varepsilon}\|_{L^{2}(\R^n)}^2\right)\nonumber\\
 =:& \mathcal F_1(\rho_{\varepsilon})+\mathcal F_2(c_{\varepsilon}).
\end{align*}
By Sobolev's inequality \eqref{Sobolev}, we know that $\mathcal F_2(c_{\varepsilon})\geq 0$. Now we only need to estimate $\mathcal F_1(\rho_{\varepsilon})$. Since $1<\frac{2n}{n+2}<m$, by the interpolation inequality we have
$$
\|\rho_{\varepsilon}\|_{L^{\frac{2n}{n+2}}(\R^n)}\leq \|\rho_{\varepsilon}\|^{1-\theta}_{L^{1}(\R^n)}\|\rho_{\varepsilon}\|^{\theta}_{L^{m}(\R^n)}, \quad \theta=\frac{(n-2)m}{2n(m-1)}.
$$
Hence
\begin{align}
\nonumber \mathcal F_1(\rho_{\varepsilon})&=\dfrac{1}{m-1}\|\rho_{\varepsilon}\|_{L^{m}(\R^n)}^m -\dfrac{1}{2S_n} \|\rho_{\varepsilon}\|_{L^{\frac{2n}{n+2}}(\R^n)}^2 \nonumber\\
&\geq  \dfrac{1}{m-1} M_0^{\frac{2n-m(n+2)}{n-2}} \|\rho_{\varepsilon}\|_{L^{\frac{2n}{n+2}}(\R^n)}^{\frac{2n(m-1)}{n-2}} -\dfrac{1}{2S_n} \|\rho_{\varepsilon}\|_{L^{\frac{2n}{n+2}}(\R^n)}^2.\label{lowerbdF1}
\end{align}
Denote
\begin{eqnarray} \label{3}
f(s):=\frac{1}{m-1} M_0^{\frac{2n-m(n+2)}{n-2}} s-\dfrac{1}{2S_n} s^{\frac{n-2}{n(m-1)}}.
\end{eqnarray}
We are now ready to utilize the decomposition of the free energy in obtaining the uniform bound of the $L^{\frac{2n}{n+2}}$ norm of $\rho_{\varepsilon}$.

The assumption $\mathcal F(\rho_{0},c_{0})<f(s^*)$ in Theorem \ref{thm} implies that there is a small number $0<\eta\leq f(s^*)-\mathcal F(\rho_{0},c_{0})$ such that $\mathcal F(\rho_{0},c_{0})\leq f(s^*)-\eta$. Hence for $0<\varepsilon\ll 1$, we have $\mathcal F(\rho_{0,\varepsilon},c_{0,\varepsilon})\leq f(s^*)-\frac{\eta}{2}$.

\begin{lem}\label{4}
Assume that $\rho_{0,\varepsilon}\in L^{1}(\R^n)$ and $\mathcal F_{\varepsilon}(\rho_{0,\varepsilon},c_{0,\varepsilon})\leq f(s^*)-\frac{\eta}{2}$, then the following holds
\begin{enumerate}[(i)]
\item If $\|\rho_{0,\varepsilon}\|_{L^{\frac{2n}{n+2}}(\R^n)}<(s^*)^{\frac{n-2}{2n(m-1)}}$, then $\exists\mu_1<1$ such that
\begin{eqnarray}\label{upbound}
\|\rho_{\varepsilon}(\cdot,t)\|_{L^{\frac{2n}{n+2}}(\R^n)} <(\mu_1s^*)^{\frac{n-2}{2n(m-1)}}, \quad \forall t>0;
\end{eqnarray}
\item If $\|\rho_{0,\varepsilon}\|_{\frac{2n}{n+2}}>(s^*)^{\frac{n-2}{2n(m-1)}}$, then $\exists\mu_2>1$ such that
\begin{eqnarray*}
\|\rho_{\varepsilon}(\cdot,t)\|_{L^{\frac{2n}{n+2}}(\R^n)} >(\mu_2s^*)^{\frac{n-2}{2n(m-1)}}, \quad \forall t>0,
\end{eqnarray*}
\end{enumerate}
where $s^*$ is the maximum point of the function $f(s)$.
\end{lem}

\begin{proof}

By (\ref{lowerbdF1}) and (\ref{3}), we have
\begin{eqnarray}\label{lowerbdF11}
\mathcal F_1(\rho_{\varepsilon})\geq f\left(\|\rho_{\varepsilon}\|_{L^{\frac{2n}{n+2}}}^{\frac{2n(m-1)}{n-2}}\right).
\end{eqnarray}
Notice that $\frac{2n}{2+n}<m<2-\frac{2}{n}$ implies $\frac{n-2}{n(m-1)}>1$. Hence we know that $f(s)$ is a strictly concave function in $0<s<\infty$, and the derivation of $f(s)$ is given by
\begin{eqnarray*}
f'(s)=\frac{1}{m-1} M_0^{\frac{2n-m(n+2)}{n-2}} -\dfrac{1}{2S_n} \frac{n-2}{n(m-1)}s^{\frac{n-2}{n(m-1)}-1}.
\end{eqnarray*}
Thus
\begin{eqnarray*}
s^{*}=\left(2S_n\frac{n}{n-2} M_0^{\frac{2n-m(n+2)}{n-2}}\right)^{\frac{n(m-1)}{2n-2-mn}}
\end{eqnarray*}
is the maximum point of $f(s)$, and $f(s)$ is monotone increasing for $0<s<s^{*}$, while $f(s)$ is monotone decreasing for $s>s^{*}$.

In the case that initial free energy $\mathcal F_{\varepsilon}(\rho_{0,\varepsilon},c_{\varepsilon})\leq f(s^{*})-\frac{\eta}{2}$, there is a $0<\delta<1$ such that $\mathcal F_{\varepsilon}(\rho_{0,\varepsilon},c_{0,\varepsilon})<\delta f(s^{*})$.

Using (\ref{lowerbdF11}), the decomposition of the free energy, the monotonicity of the free energy and our assumptions, we have
\begin{eqnarray*}
f(\|\rho_{\varepsilon}\|_{L^{\frac{2n}{n+2}}}^{\frac{2n(m-1)}{n-2}})\leq \mathcal F_1(\rho_{\varepsilon})\leq \mathcal F_{\varepsilon}(\rho_{\varepsilon},c_{\varepsilon})\leq \mathcal F_{\varepsilon}(\rho_{0,\varepsilon},c_{0,\varepsilon})\leq\delta f(s^{*}).
\end{eqnarray*}
If $\|\rho_{0,\varepsilon}\|_{L^{\frac{2n}{n+2}}(\R^n)}<(s^*)^{\frac{n-2}{2n(m-1)}}$, due to the fact that $f(s)$ is increasing in $0<s<s^{*}$, there exists a $\mu_{1}<1$ such that
\begin{eqnarray*}
 \|\rho_{\varepsilon}(\cdot,t)\|_{L^{\frac{2n}{n+2}}(\R^n)} <(\mu_1s^*)^{\frac{n-2}{2n(m-1)}}, \quad \forall t>0.
\end{eqnarray*}

Inversely, if $\|\rho_0\|^{\frac{2n(m-1)}{n-2}}_{L^{\frac{2n}{n+2}}}>s^*$, then the increasing property of $f(s)$ in $s>s^*$ implies that there exists a constant $\mu_2>1$ such that  $\|\rho_{\varepsilon}(\cdot,t)\|_{L^{\frac{2n}{n+2}}}>(\mu_2 s^*)^{\frac{n-2}{2n(m-1)}}$.
\end{proof}
\begin{rem}
Lemma \ref{4} gives us a hint that $(s^*)^{\frac{n-2}{2n(m-1)}}$ would be expected to be the sharp initial criterion for parabolic-parabolic Keller-Segel model (\ref{equ1})-(\ref{initialdata}). The fact has been proved in the parabolic-elliptic case in \cite{CW}.
\end{rem}

\subsection{Uniform estimates}
We focus on the uniform estimates of the regularized solutions in this
subsection. Using the $L^{\frac{2n}{n+2}}$ bound of $\rho_{\varepsilon}$ obtained in above subsection, we
show the uniform $L^p$ estimate by using standard method. Furthermore, the
uniform estimates for space and time derivatives will be derived carefully.

Now we prove the uniform $L^{p}$ estimate of regularized solutions $\rho_{\varepsilon}$ for any $p>1$.
\begin{lem}\label{5}
Assume that $\rho_{0,\varepsilon}\in L^{1}(\R^n)\cap L^{p}(\R^n)$, $p>1$, $\|\rho_{0,\varepsilon}\|_{L^{\frac{2n}{n+2}}(\R^n)}<(s^{*})^{\frac{n-2}{2n(m-1)}}$, $\mathcal F_{\varepsilon}(\rho_{0,\varepsilon},c_{0,\varepsilon})\leq f(s^{*})-\frac{\eta}{2}$ and $\nabla c_{0,\varepsilon}\in L^{\infty}(\mathbb{R}^n)$. Let $\rho_{\varepsilon}$ be a smooth solution of the regularized problem. Then
\begin{eqnarray}
&&\|\rho_{\varepsilon}\|_{L^{\infty}(0,\infty;L^{p}(\R^n))}\leq C,\label{Lp}\\
&&\|\nabla c_{\varepsilon}\|_{L^\infty(0,\infty; L^{\gamma}(\mathbb{R}^n))}\leq C, \quad 1\leq \gamma\leq \infty.\label{nablac}
\end{eqnarray}
Moreover, for any fixed $T>0$, it holds that
\begin{eqnarray}
&&\|\rho_{\varepsilon}\|_{L^{p+1}(0,T;L^{p+1}(\R^n))}\leq C,\label{p+1}\\
&&\big\|\nabla\rho_{\varepsilon}^{\frac{m+p-1}{2}}\big\|_{L^{2}(0,T;L^{2}(\R^n))}\leq C.\label{nablarho}
\end{eqnarray}
Here $C$ is a constant independent of $\varepsilon$.
\end{lem}
\begin{proof}
Multiplying the equation (\ref{regequ1}) by $p\rho_{\varepsilon}^{p-1}$ with $p>1$ and integrating it in space variable, we have for any  $t>0$
\begin{align}
 \frac{d}{dt}\dint_{\R^n}\rho_{\varepsilon}^p(x,t) dx
=&-pm(p-1)\dint_{\R^n}(\rho_{\varepsilon}+\varepsilon)^{m-1}\rho_{\varepsilon}^{p-2}|\nabla\rho_{\varepsilon}|^{2}dx\label{lpest1}\\
&+(p-1)\dint_{\R^n}(\nabla \rho_{\varepsilon}^{p}*J_{\varepsilon})\cdot\nabla c_{\varepsilon}dx
+p\varepsilon\dint_{\R^n}(\nabla \rho_{\varepsilon}^{p-1}*J_{\varepsilon})\cdot\nabla c_{\varepsilon}dx.\nonumber
\end{align}
Noticing that $(\rho_{\varepsilon}+\varepsilon)^{m-1}>\rho_{\varepsilon}^{m-1}$ due to $m>1$ and using the integration by parts, we get from (\ref{lpest1}) that
\begin{align}
 \frac{d}{dt}\dint_{\R^n}\rho_{\varepsilon}^p(x,t) dx
 \leq&-pm(p-1)\dint_{\R^n}\rho_{\varepsilon}^{m+p-3}|\nabla\rho_{\varepsilon}|^{2}dx
+p\varepsilon\|\nabla c_{\varepsilon}\|_{L^{\infty}(\R^n)}\dint_{\R^n}|\nabla \rho_{\varepsilon}^{p-1}*J_{\varepsilon}|\,dx\nonumber\\
&+(p-1)\|\nabla c_{\varepsilon}\|_{L^{\infty}(\R^n)}\dint_{\R^n}|\nabla \rho_{\varepsilon}^{p}*J_{\varepsilon}|\,dx. \label{lpest2}
\end{align}
Using the H\"{o}lder inequality and a series of computations, we have
\begin{align}
\dint_{\R^n}|\nabla \rho_{\varepsilon}^{p-1}*J_{\varepsilon}|\,dx\label{converlution}
&\leq \dint_{\R^n}|\nabla \rho_{\varepsilon}^{p-1}|\,dx\nonumber\\
&=(p-1)\dint_{\R^n}\rho_{\varepsilon}^{p-2-\frac{m+p-3}{2}}\rho_{\varepsilon}^{\frac{m+p-3}{2}}|\nabla \rho_{\varepsilon}|\,dx\nonumber\\
&\leq (p-1)\big\|\rho_{\varepsilon}\big\|^{\frac{p-1-m}{2}}_{L^{p-1-m}(\R^n)}\big\|\rho_{\varepsilon}^{\frac{m+p-3}{2}}\nabla \rho_{\varepsilon}\big\|_{L^2(\R^n)}.
\end{align}
Similarly, we have
\begin{eqnarray}\label{converlution1}
\dint_{\R^n}|\nabla \rho_{\varepsilon}^{p}*J_{\varepsilon}|\,dx
\leq p\big\|\rho_{\varepsilon}\big\|^{\frac{p+1-m}{2}}_{L^{p+1-m}(\R^n)}\big\|\rho_{\varepsilon}^{\frac{m+p-3}{2}}\nabla \rho_{\varepsilon}\big\|_{L^2(\R^n)}.
\end{eqnarray}
Hence substituting (\ref{converlution}) and (\ref{converlution1}) into (\ref{lpest2}), we obtain
\begin{align}
 &\frac{d}{dt}\dint_{\R^n}\rho_{\varepsilon}^p(x,t) dx\nonumber\\
 \leq
&p(p-1)\|\nabla c_{\varepsilon}\|_{L^{\infty}(\R^n)}\big\|\rho_{\varepsilon}^{\frac{m+p-3}{2}}\nabla \rho_{\varepsilon}\big\|_{L^2(\R^n)}
\left(\varepsilon\big\|\rho_{\varepsilon}\big\|^{\frac{p-1-m}{2}}_{L^{p-1-m}(\R^n)}
+\big\|\rho_{\varepsilon}\big\|^{\frac{p+1-m}{2}}_{L^{p+1-m}(\R^n)}\right) \nonumber\\
&-pm(p-1)\dint_{\R^n}\rho_{\varepsilon}^{m+p-3}|\nabla\rho_{\varepsilon}|^{2}dx.\label{lpest3}
\end{align}
Using the Young inequality for the last term of (\ref{lpest3}), we deduce for any $\nu>0$
 \begin{align}
 \frac{d}{dt}\dint_{\R^n}\rho_{\varepsilon}^p(x,t) dx
 \leq&+\frac{p^2(p-1)^2
}{4\nu}\|\nabla c_{\varepsilon}\|^2_{L^{\infty}(\R^n)}
\left(\varepsilon\big\|\rho_{\varepsilon}\big\|^{\frac{p-1-m}{2}}_{L^{p-1-m}(\R^n)}
+\big\|\rho_{\varepsilon}\big\|^{\frac{p+1-m}{2}}_{L^{p+1-m}(\R^n)}\right)^2\nonumber\\
&-pm(p-1)\dint_{\R^n}\rho_{\varepsilon}^{m+p-3}|\nabla\rho_{\varepsilon}|^{2}dx+\nu\big\|\rho_{\varepsilon}^{\frac{m+p-3}{2}}\nabla \rho_{\varepsilon}\big\|^2_{L^2(\R^n)}.  \label{lpest4}
\end{align}
Taking $\nu=\frac{pm(p-1)}{4}$ in (\ref{lpest4}), we get
\begin{align}
 \frac{d}{dt}\dint_{\R^n}\rho_{\varepsilon}^p(x,t) dx
\leq &-\frac{3pm(p-1)}{(m+p-1)^2}\dint_{\R^n}\big|\nabla\rho^{\frac{m+p-1}{2}}_{\varepsilon}\big|^{2}dx \label{lpest5} \\
&+\frac{2p(p-1)}{m}\|\nabla c_{\varepsilon}\|^2_{L^{\infty}}
\left(\varepsilon^2\|\rho_{\varepsilon}\|^{p-1-m}_{L^{p-1-m}(\R^n)}+\|\rho_{\varepsilon}\|^{p+1-m}_{L^{p+1-m}(\R^n)}\right).\nonumber
\end{align}
Notice the following facts:
\begin{enumerate}[(i)]
\item if $p>1+m+\frac{2n}{n+2}$, then
\begin{eqnarray}
\|\rho_{\varepsilon}\|_{L^{p-1-m}(\R^n)}\leq (S_n)^{-\frac{\theta_1}{m+p-1}}\|\rho_{\varepsilon}\|^{1-\theta_1}_{L^{\frac{2n}{n+2}}(\R^n)}
\big\|\nabla\rho_{\varepsilon}^{\frac{m+p-1}{2}}\big\|^{\frac{2\theta_1}{m+p-1}}_{L^2(\R^n)},\label{formula3-1}
\end{eqnarray}
where $\theta_1=\left(\frac{n+2}{2n}-\frac{1}{p-1-m}\right)\Big/\left(\frac{n+2}{2n}-\frac{n-2}{n(m+p-1)}\right)$;
\item if $p>m+\frac{n-2}{n+2}$, then
\begin{eqnarray}
\|\rho_{\varepsilon}\|_{L^{p+1-m}(\R^n)}\leq (S_n)^{-\frac{\theta_2}{m+p-1}}\|\rho_{\varepsilon}\|^{1-\theta_2}_{L^{\frac{2n}{n+2}}(\R^n)}
\big\|\nabla\rho_{\varepsilon}^{\frac{m+p-1}{2}}\big\|^{\frac{2\theta_2}{m+p-1}}_{L^2(\R^n)},\label{formula3+1}
\end{eqnarray}
where $\theta_2=\left(\frac{n+2}{2n}-\frac{1}{p+1-m}\right)\Big/\left(\frac{n+2}{2n}-\frac{n-2}{n(m+p-1)}\right)$.
\end{enumerate}
Using the estimates (\ref{formula3-1}), (\ref{formula3+1}) and the uniform in time upper bound (\ref{upbound}) of $\|\rho_{\varepsilon}\|_{L^{\frac{2n}{n+2}}}$, we deduce form (\ref{lpest5}) that
\begin{align}\label{estimate}
 &\frac{d}{dt}\dint_{\R^n}\rho_{\varepsilon}^p(x,t) dx\nonumber\\
\leq &-\frac{3pm(p-1)}{(m+p-1)^2}\dint_{\R^n}\big|\nabla\rho^{\frac{m+p-1}{2}}_{\varepsilon}\big|^{2}dx \nonumber \\
&+C(m,n,p)\|\nabla c_{\varepsilon}\|^2_{L^{\infty}(\R^n)}
\left(\big\|\nabla\rho_{\varepsilon}^{\frac{m+p-1}{2}}\big\|^{\frac{2\theta_1(p-1-m)}{m+p-1}}_{L^2(\R^n)}
+\big\|\nabla\rho_{\varepsilon}^{\frac{m+p-1}{2}}\big\|^{\frac{2\theta_2(p+1-m)}{m+p-1}}_{L^2(\R^n)}\right).
\end{align}
Since $0<\theta_1,\theta_2<1$ and $m>1$, we know that
$$
0<\frac{\theta_1(p-1-m)}{m+p-1}, ~~\frac{\theta_2(p+1-m)}{m+p-1}<1,
$$
which allows us to utilize the Young inequality for the last two terms of (\ref{estimate}) such that it holds that
\begin{align}\label{estimate1}
 \frac{d}{dt}\dint_{\R^n}\rho_{\varepsilon}^p(x,t) dx
\leq &-\frac{3pm(p-1)}{(m+p-1)^2}\dint_{\R^n}\big|\nabla\rho^{\frac{m+p-1}{2}}_{\varepsilon}\big|^{2}dx \nonumber \\
&+C(m,n,p)\|\nabla c_{\varepsilon}\|^{2q_1}_{L^{\infty}(\R^n)}
+\sigma_1\big\|\nabla\rho_{\varepsilon}^{\frac{m+p-1}{2}}\big\|^{\frac{2\theta_1(p-1-m)q_2}{m+p-1}}_{L^2(\R^n)}\nonumber\\
&+C(m,n,p)\|\nabla c_{\varepsilon}\|^{2\ell_1}_{L^{\infty}(\R^n)}
+\sigma_2\big\|\nabla\rho_{\varepsilon}^{\frac{m+p-1}{2}}\big\|^{\frac{2\theta_2(p+1-m)\ell_2}{m+p-1}}_{L^2(\R^n)},
\end{align}
where $q_1,q_2>1$ satisfy $\frac{1}{q_1}+\frac{1}{q_2}=1$, $\ell_1,\ell_2>1$ satisfy $\frac{1}{\ell_1}+\frac{1}{\ell_2}=1$. Setting
\begin{eqnarray}\label{constant2}
\sigma_1=\sigma_2=\frac{pm(p-1)}{(m+p-1)^2}, \quad \frac{2\theta_1(p-1-m)q_2}{m+p-1}=2,\quad \frac{2\theta_2(p+1-m)\ell_2}{m+p-1}=2,
\end{eqnarray}
from (\ref{estimate1}) we have that
\begin{align}\label{estimate2}
 &\frac{d}{dt}\int_{\R^n}\rho_{\varepsilon}^p(x,t) dx\nonumber\\
\leq &-C_1\big\|\nabla\rho^{\frac{m+p-1}{2}}_{\varepsilon}\big\|_{L^2(\R^n)}^{2}+C(m,n,p)\Big(\|\nabla c_{\varepsilon}\|^{2q_1}_{L^{\infty}(\R^n)}+\|\nabla c_{\varepsilon}\|^{2\ell_1}_{L^{\infty}(\R^n)}\Big),
\end{align}
where $0<C_1<\frac{pm(p-1)}{(m+p-1)^2}$ is a constant independent of $p$.

On the other hand, using (\ref{nablacLinfty}) for $c_{\varepsilon}$ in (\ref{regequ2}) and (\ref{upbound}), we have
\begin{align}
\|\nabla c_{\varepsilon}\|^2_{L^{\infty}(\R^n)}&\leq 2\left(\|\nabla c_0\|^2_{L^{\infty}(\R^n)}+\left(B_{\infty,q,n} ~\Gamma \left(\frac{1}{2}-\frac{n}{2q}\right)\right)^2\|\rho_{\varepsilon}\|^2_{L^{\infty}(0,\infty;~L^q(\R^n))}\right)\nonumber\\
&\leq 2\|\nabla c_0\|^2_{L^{\infty}(\R^n)}+C(m,n,q)\|\rho_{\varepsilon}\|^{2(1-\theta_3)}_{L^{\infty}\big(0,\infty; L^{\frac{2n}{n+2}}(\R^n)\big)}
\|\rho_{\varepsilon}\|^{2\theta_3}_{L^{\infty}(0,\infty;~L^p(\R^n))}\nonumber\\
&\leq 2\|\nabla c_0\|^2_{L^{\infty}(\R^n)}+C(m,n,p)
\|\rho_{\varepsilon}\|^{2\theta_3}_{L^{\infty}(0,\infty;~L^p(\R^n))},\label{ccontrol}
\end{align}
where $n<q<p$ is an exponent independent of $p$, which will be chosen later, and $\theta_3$ satisfies
\begin{align}\label{theta3}
\theta_3=\left(\frac{n+2}{2n}-\frac{1}{q}\right)\Big/\left(\frac{n+2}{2n}-\frac{1}{p}\right).
\end{align}
Hence (\ref{estimate2}) and (\ref{ccontrol}) imply that
\begin{align}\label{estimate3}
 \frac{d}{dt}\dint_{\R^n}\rho_{\varepsilon}^p(x,t) dx
\leq& -C_1\big\|\nabla\rho^{\frac{m+p-1}{2}}_{\varepsilon}\big\|_{L^2(\R^n)}^{2}+C(m,n,p)\Big(\|\nabla c_0\|^{2q_1}_{L^{\infty}(\R^n)}+\|\nabla c_0\|^{2\ell_1}_{L^{\infty}(\R^n)}\Big)\nonumber\\
&+C(m,n,p)\Big(\|\rho_{\varepsilon}\|^{2q_1\theta_3}_{L^{\infty}(0,\infty; L^p(\R^n))}+\|\rho_{\varepsilon}\|^{2\ell_1\theta_3}_{L^{\infty}(0,\infty; L^p(\R^n))}\Big).
\end{align}

From (\ref{constant2}), it can be computed that
\begin{align}\label{ell1}
\ell_1=\frac{m+p-1}{(m+p-1)-\theta_2(p+1-m)},~~ \mbox{and } 1<q_1<\ell_1.
\end{align}
Hence there is a constant $C(m,n,p)>1$ such that (\ref{estimate3}) can be recast as
\begin{align}\label{estimate4}
 \frac{d}{dt}\dint_{\R^n}\rho_{\varepsilon}^p(x,t) dx
\leq& -C_1\big\|\nabla\rho^{\frac{m+p-1}{2}}_{\varepsilon}\big\|_{L^2(\R^n)}^{2}+C(m,n,p)
\Big(1+\|\rho_{\varepsilon}\|^{2\ell_1\theta_3}_{L^{\infty}(0,\infty; L^p(\R^n))}\Big).
\end{align}
Due to $\frac{2n}{n+2}<m<2-\frac{2}{n}$, we have $n<\frac{2n}{n-(m-1)(n+2)}$. Thus taking $n<q<\frac{2n}{n-(m-1)(n+2)}<p$, and using (\ref{ell1}) and the expression (\ref{theta3}) of $\theta_3$, we can derive that there is $p_0$ such that when $p>\max\{\frac{2n}{n-(m-1)(n+2)},p_0\}$, it holds that
\begin{align}\label{xiaop}
2\ell_1\theta_3=\frac{[(n+2)(m+p-1)-2(n-2)](n+2-2n/q)}{[(n+2)(m-1)+2](n+2-2n/p)}<p.
\end{align}

Furthermore, taking~$p>\frac{2n}{n+2}$, it is easily known that~$p<\frac{n(m+p-1)}{n-2}$. So, using the interpolation inequality, Sobolev's inequality, one has that
\begin{eqnarray}
\|\rho_{\varepsilon}\|^p_{L^{p}(\R^n)}&\leq &S_{n}^{-\frac{p\theta_4}{m+p-1}}\|\rho_{\varepsilon}\|_{L^{\frac{2n}{n+2}}(\R^n)}^{p(1-\theta_4)}
\big\|\nabla\rho_{\varepsilon}^{\frac{m+p-1}{2}}\big\|_{L^{2}(\R^n)}^{\frac{2\theta_4 p}{m+p-1}},
\label{29}
\end{eqnarray}
where $\theta_4$ satisfies~$\frac{1}{p}=\frac{(n+2)(1-\theta_4)}{2n}+\frac{(n-2)\theta_4}{n(m+p-1)}$.
 Noticing $\frac{2\theta_3 p}{m+p-1}<2$, and using (\ref{upbound}) and the Young inequality for (\ref{29}), it holds that for any $\nu_{2}>0$
\begin{eqnarray}\label{lp1}
\|\rho_{\varepsilon}\|_{L^{p}(\R^n)}^{p}&\leq&C(n,m,p)
+\nu_2\big\|\nabla\rho_{\varepsilon}^{\frac{m+p-1}{2}}\big\|_{L^{2}(\R^n)}^{2}.
\end{eqnarray}
Taking $\nu_{2}=C_1$, from~(\ref{estimate4})~and~(\ref{lp1}), we have
\begin{align}\label{estimate5}
 \frac{d}{dt}\|\rho_{\varepsilon}\|^p_{L^p(\R^n)}
\leq& -\|\rho_{\varepsilon}\|^p_{L^p(\R^n)}+C(m,n,p)
\Big(1+\|\rho_{\varepsilon}\|^{2\ell_1\theta_3}_{L^{\infty}(0,\infty; L^p(\R^n))}\Big).
\end{align}
Solving the above ordinary differential inequality (\ref{estimate5}), we obtain
\begin{align}\label{estimate6}
 \|\rho_{\varepsilon}\|^p_{L^p(\R^n)}
\leq& \|\rho_{\varepsilon0}\|^p_{L^p(\R^n)}+C(m,n,p)
\Big(1+\|\rho_{\varepsilon}\|^{2\ell_1\theta_3}_{L^{\infty}(0,\infty; L^p(\R^n))}\Big).
\end{align}
By (\ref{xiaop}), and taking the supremum of (\ref{estimate6}) for $t\in [0,\infty)$, we know that
$$
\|\rho_{\varepsilon}\|_{L^{\infty}(0,\infty; L^p(\R^n))}\leq C(m,n,p),\quad p>\max\{1+m+\frac{2n}{n+2},\frac{2n}{n-(m-1)(n+2)},p_0\}.
$$
Due to the conservation of mass, the interpolation inequality implies that for any $p>1$, (\ref{Lp}) holds true.

Furthermore, by (\ref{estimate4}) we deduce that for any fixed $T>0$, it holds that
\begin{eqnarray}\label{28}
 \sup_{t\in[0,T)}\dint_{\R^n}\rho_{\varepsilon}^p(x,t)dx +\frac{2pm(p-1)}{(m+p-1)^{2}}\int^T_0\dint_{\R^n}\big|\nabla\rho_{\varepsilon}^{\frac{m+p-1}{2}}\big|^{2}dxdt
 \leq C(T),
\end{eqnarray}
Thus for~$p>1$, we obtain that
\begin{eqnarray*}
\big\|\nabla\rho_{\varepsilon}^{\frac{m+p-1}{2}}\big\|_{L^{2}(0,T;L^{2}(\R^n))}\leq C(T).
\end{eqnarray*}

 Due to $p>1$, it can be easily checked that $p+1<\frac{n(m+p-1)}{n-2}$. Then the interpolation inequality tells us that
\begin{align}
\|\rho_{\varepsilon}\|_{L^{p+1}(\R^n)}&\leq \|\rho_{\varepsilon}\|_{L^{\frac{2n}{n+2}}(\R^n)}^{1-\theta}\|\rho_{\varepsilon}\|_{L^{\frac{n(m+p-1)}{n-2}}(\R^n)}^{\theta}\nonumber\\
&=\|\rho_{\varepsilon}\|_{L^{\frac{2n}{n+2}}(\R^n)}^{1-\theta}\big\|\rho_{\varepsilon}^{\frac{m+p-1}{2}}\big\|_{L^{\frac{2n}{n-2}}(\R^n)}^{\frac{2}{m+p-1}\theta},\label{ffc}
\end{align}
where~$\frac{1}{p+1}=\frac{(n+2)(1-\theta)}{2n}+\frac{(n-2)\theta}{n(m+p-1)}$. Using the Sobolev inequality for (\ref{ffc}), it follows that
\begin{eqnarray}
\|\rho_{\varepsilon}\|^{p+1}_{L^{p+1}(\R^n)}&\leq &S_{n}^{-\frac{(p+1)\theta}{m+p-1}}\|\rho_{\varepsilon}\|_{L^{\frac{2n}{n+2}}(\R^n)}^{(p+1)(1-\theta)}
\big\|\nabla\rho_{\varepsilon}^{\frac{m+p-1}{2}}\big\|_{L^{2}(\R^n)}^{\frac{2\theta (p+1)}{m+p-1}}.
\label{291}
\end{eqnarray}
Let $k:=\frac{2\theta (p+1)}{m+p-1}$. A simple computation shows $k<2$. Hence utilizing (\ref{nablarho}) and the Young inequality for (\ref{291}), we get
\begin{eqnarray*}
\|\rho_{\varepsilon}\|_{L^{p+1}(0,T;L^{p+1}(\R^n))}\leq C(T).
\end{eqnarray*}
This is the proof on (\ref{Lp}), (\ref{nablarho}) and (\ref{p+1}). In addition, we obtain easily that (\ref{nablac}) is a direct consequence of (\ref{nablacLp}) and (\ref{Lp}).

\end{proof}

\begin{rem}
 Let $ p=m$, we get the estimate used later
\begin{eqnarray*}\label{m}
 \|\rho_{\varepsilon}\|_{L^{\infty}(0,\infty;L^{m}(\R^n))}+ \|\rho_{\varepsilon}\|_{L^{m+1}(0,T;L^{m+1}(\R^n))}\leq C.
\end{eqnarray*}

\end{rem}

\subsection{Uniform estimates for the space and time derivatives}

The estimates on space and time derivative of $\rho_{\varepsilon}$ are two necessary conditions for compactness arguments. First, we will use the $L^p$ estimate that obtained above to prove the estimate on the space derivative.

\begin{lem}\label{lem3} Assume that the assumptions of Lemma \ref{5} hold, then for any fixed $T>0$, there exists a constant $C>0$ independent of $\varepsilon$ such that
\begin{eqnarray*}
&&\|\nabla\rho_{\varepsilon}\|_{L^{2}(0,T; L^{\frac{2m}{3-m}}(\mathbb{R}^n))}\leq C, \quad \mbox{ for } m< \frac{3}{2}, \label{est28}\\
&&\|\nabla\rho_{\varepsilon}\|_{L^{2}(0,T; L^{2}(\mathbb{R}^n))}\leq C, \quad \mbox{ for } m\geq \frac{3}{2}.\label{est29}
\end{eqnarray*}
\end{lem}

\begin{proof}
The proof of Lemma \ref{lem3} is the same to \cite[Lemma 2.3]{CW}. We omit it.
\end{proof}

Now, we give the estimate of the time derivative of $\rho_{\varepsilon}$.½Ó
\begin{lem}\label{timed}
Assume that the assumptions of Lemma \ref{5} hold, then for any fixed $T>0$, there exists a constant $C>0$ independent of $\varepsilon$ such that
\begin{eqnarray}
&&\|\partial_t\rho_{\varepsilon}\|_{L^2(0,T; (W_{\rm loc}^{1,s'}(\mathbb{R}^n))')}\leq C,\label{19}\\
&&\|\partial_{t}c_{\varepsilon}\|_{L^{\infty}(0,T;W^{-1,2}(\R^n))}\leq C,\label{partialc}
\end{eqnarray}
where $s'$ satisfies $\frac{1}{s'}+\frac{1}{s}=1$ and $s=\min\{\dfrac{2m}{m+1},\dfrac{nm(m+1)}{nm+(n-m)(m+1)}\}>1$.
\end{lem}
\begin{proof}

The proof of (\ref{19}) is the same as that in \cite[Lemma 2.4]{CW}. Using the second equation (\ref{regequ2}) with (\ref{Lp}), (\ref{cLp}) and (\ref{nablacLp}), the estimate (\ref{partialc}) can be proved.

\end{proof}

\subsection{Compactness argument and the proof of Theorem \ref{thm}}

Utilizing the uniform estimates in above two subsections and the Lions-Aubin Lemma \cite{CYL,L}, We have the following convergence.
\begin{lem}\label{lem1}
Assume that $(\rho_0,~c_0)$ satisfies the assumptions of Theorem \ref{thm}. Let~$(\rho_{\varepsilon},c_{\varepsilon})$~be the solution to (\ref{regequ1})-(\ref{reginitial}). Then there is a subsequence of $\{\rho_{\varepsilon}\}$,~$\{c_{\varepsilon}\}$(without relabeling for convenience) and functions $\rho$ and $c$ such that as $\varepsilon\to 0$
\begin{eqnarray}
&&\rho_{\varepsilon}\stackrel{\text{*}}{\rightharpoonup}\rho   \quad\mbox{in}~L^{\infty}(0,T;L^{1}\cap L^{m}(\R^n)),\label{conver1}\\
&&\rho_{\varepsilon}\rightharpoonup\rho    \quad\mbox{in}~L^{m+1}(0,T;L^{m+1}(\R^n)),\\
&&\nabla\rho_{\varepsilon}\rightharpoonup\nabla\rho    \quad\mbox{in}~L^{2}(0,T;L^{r}(\R^n)),\\
&&\partial_{t}\rho_{\varepsilon}\rightharpoonup\partial_{t}\rho  \quad\mbox{in}~L^{2}(0,T;(W_{loc}^{1,s'}(\R^n))')),\\
&&\partial_{t}c_{\varepsilon}\stackrel{\text{*}}{\rightharpoonup}\partial_{t}c    \quad\mbox{in}~L^{\infty}(0,T;W^{-1,2}(\R^n)),\\
&&\nabla c_{\varepsilon}\stackrel{\text{*}}{\rightharpoonup}\nabla c           \quad\mbox{in}~{L^\infty(0,T; L^l(\mathbb{R}^n))},\label{conver5}
\end{eqnarray}
where~$r=\min\{2,\frac{2m}{3-m}\}$, $\frac{1}{s'}+\frac{1}{s}=1$ and $s=\min\{\dfrac{2m}{m+1},\dfrac{nm(m+1)}{nm+(n-m)(m+1)}\}>1$,~$1\leq l <\infty$.
Moreover, the following strong convergence holds
\begin{eqnarray}\label{strong}
\rho_{\varepsilon}\rightarrow\rho \quad \mbox{in}~ L^{2}(0,T;L^{2}(\R^n)).
\end{eqnarray}
\end{lem}
\begin{proof}
From Lemma \ref{5}, Lemma \ref{lem3} and Lemma \ref{timed}, we deduce that the convergence results in (\ref{conver1})-(\ref{conver5}) hold true.

Using the Lions-Aubin lemma and the estimates (\ref{28})-(\ref{19}), we obtain (\ref{strong}).
\end{proof}

\begin{proof}[\bf The proof of Theorem \ref{thm}]

With the help of the convergence results in Lemma \ref{lem1}, the existence of weak solutions can be obtained directly by taking the limit $\varepsilon\to 0$ in the weak formulation of the regularized problem. Namely, for any $\varphi\in C_{c}^{\infty}(\R^n)$, $(\rho_{\varepsilon},c_{\varepsilon})$~satisfies the following equations
\begin{align*}
\dint^t_0\dint_{\R^n}\partial_{t}\rho_{\varepsilon}\varphi dxdt
&=-\dint^t_0\dint_{\R^n}\nabla\varphi\nabla(\rho_{\varepsilon}+\varepsilon)^m dxdt+\dint^t_0\dint_{\R^n}\nabla\varphi(\rho_{\varepsilon}+\varepsilon)\nabla c_{\varepsilon}\ast J_{\varepsilon} dxdt,\\
\dint^t_0\dint_{\R^n}\partial_{t}c_{\varepsilon}\varphi dxdt&=-\int_{0}^{t}\int_{\mathbb R^n} \nabla c_{\varepsilon} \nabla \varphi dx ds-\int_{0}^{t}\int_{\mathbb R^n} c_{\varepsilon}  \varphi dx ds\nonumber+\int_{0}^{t}\int_{\mathbb R^n} \rho_{\varepsilon} \varphi dx ds.
\end{align*}
For $\varepsilon\to 0$, we can prove that the limit function $(\rho,c)$~satisfies~(\ref{r1})~and~(\ref{r2}), i.e., $(\rho,c)$ is a weak solution of the model (\ref{equ1})-(\ref{initialdata}) in the distribution sense.
\end{proof}

 \section{The uniform in time $L^{\infty}$-bound of weak solutions}
 In this section we prove that weak solutions to (\ref{equ1})-(\ref{initialdata}) have a
 uniform in time $L^\infty$ bound by utilizing a modified Moser iteration, which has been successfully applied in the parabolic-elliptic Keller-Segel system in \cite{BLZ,LW}. As that has been already explained in the introduction, the iteration strongly depends on the explicit expression of the constant $B_{\infty,q,n}$ in (\ref{nablacLinfty}). The main result in this section is the following iteration lemma.
\begin{lem}\label{lem2}(The $L^{p_{k}}$ estimate)
Assume initial density $\rho_0\in L^{\infty} (\mathbb R^{n})$. Let the assumptions of Theorem \ref{thm} hold. Set $p_k=2^k+4n+4$ for $k=1,2,\cdot\cdot\cdot$. Then the $L^{p_k}$ norm of solutions satisfies the following the inequality
\begin{align}\label{ODEinequa}
\frac{d}{dt}\|\rho\|^{p_k}_{L^{p_k}(\R^n)} \leq& -\|\rho\|^{p_k}_{L^{p_k}(\R^n)}+ C p_k^{2n}\Big(\Big(\|\rho\|_{L^{p_{k-1}}(\R^n)}^{{p_{k-1}}}\Big)^{\eta_1}
+\Big(\|\rho\|_{L^{p_{k-1}}(\R^n)}^{{p_{k-1}}}\Big)^{\eta_2}\nonumber\\
&+\Big(\|\rho\|_{L^{\infty}(0,\infty; L^{p_{k-1}}(\R^n))}^{{p_{k-1}}}\Big)^{\eta_3}\Big),
\end{align}
where $0<\eta_1,\eta_2,\eta_3\leq 2$, and $C$ is a constant independent of $p_k$.
\end{lem}

\begin{proof}
Multiplying $p_k\rho^{p_k-1}$ ($k=1,2,\cdot\cdot\cdot$) to the first equation of (\ref{equ1}) and integrating it in $\mathbb{R}^n$, we have
\begin{align}
&\frac{d}{dt}\dint_{\mathbb{R}^n} \rho^{p_k} \,dx\nonumber\\
=&-mp_k(p_k-1)\dint_{\mathbb{R}^n}\rho^{m+p_k-3}|\nabla\rho|^2 \,dx +p_k(p_k-1)\dint_{\mathbb{R}^n}\rho^{p_k-1}\nabla c\nabla\rho\,dx\nonumber\\
\leq&-mp_k(p_k-1)\dint_{\mathbb{R}^n}\rho^{m+p_k-3}|\nabla\rho|^2 \,dx +p_k(p_k-1)\|\nabla c\|_{L^\infty(\R^n)}\dint_{\mathbb{R}^n}\rho^{p_k-1}|\nabla\rho|\,dx.\label{formula}
\end{align}
Furthermore because of
\begin{eqnarray}\label{intepolation}
\dint_{\mathbb{R}^n}\rho^{p_k-1}|\nabla\rho|\,dx\leq \Big(\dint_{\mathbb{R}^n}\rho^{p_k-m+1}\,dx\Big)^{1/2}\Big(\dint_{\mathbb{R}^n}\rho^{p_k+m-3}|\nabla\rho|^2\,dx\Big)^{1/2},
\end{eqnarray}
we deduce from (\ref{formula}) and (\ref{intepolation}) that for any $\sigma_1>0$,
\begin{align}\label{dire}
\frac{d}{dt}\dint_{\mathbb{R}^n} \rho^{p_k} \,dx
&\leq-mp_k(p_k-1)\dint_{\mathbb{R}^n}\rho^{m+p_k-3}|\nabla\rho|^2 \,dx +\sigma_1\dint_{\mathbb{R}^n}\rho^{p_k+m-3}|\nabla\rho|^2\,dx\nonumber\\
&\quad +\frac{p_k^2(p_k-1)^2}{4\sigma_1} \|\nabla c\|^2_{L^\infty(\R^n)}\dint_{\mathbb{R}^n}\rho^{p_k-m+1}\,dx.
\end{align}
Taking $\sigma_1=\frac{mp_k(p_k-1)}{4}$, then (\ref{dire}) can be written as the following form
\begin{align*}
&\frac{d}{dt}\dint_{\mathbb{R}^n} \rho^{p_k} \,dx\\
=&-\dfrac{3p_km(p_k-1)}{(m+p_k-1)^2}\dint_{\mathbb{R}^n}|\nabla\rho^{\frac{m+p_k-1}{2}}|^2 \,dx +\frac{p_k(p_k-1)}{m}\|\nabla c\|^2_{L^\infty(\R^n)}\dint_{\mathbb{R}^n}\rho^{p_k-m+1}\,dx\\
\leq&-3C_1\dint_{\mathbb{R}^n}|\nabla\rho^{\frac{m+p_k-1}{2}}|^2 \,dx +\frac{p_k(p_k-1)}{m}\|\nabla c\|^2_{L^\infty(\R^n)}\dint_{\mathbb{R}^n}\rho^{p_k-m+1}\,dx,
\end{align*}
where $0<C_1\leq \dfrac{p_km(p_k-1)}{(m+p_k-1)^2}$ is a fixed constant.

Notice the following inequality
$$
\|\rho\|_{L^{p_k-m+1}(\R^n)}\leq (S_n)^{-\frac{\theta_1}{m+p_k-1}}\|\rho\|^{1-\theta_1}_{L^{p_{k-1}}(\R^n)}
\|\nabla\rho^{\frac{m+p_k-1}{2}}\|^{\frac{2\theta_1}{m+p_k-1}}_{L^2(\R^n)}
$$
with $\theta_1=\left(\frac{1}{p_{k-1}}-\frac{1}{p_k+1-m}\right)\Big/\left(\frac{1}{p_{k-1}}-\frac{n-2}{n(m+p_k-1)}\right)$.
And taking $p=p_{k-1}$ in (\ref{ccontrol}), we have $\theta_3\sim O(1)$ as $k\to +\infty$. Therefore, together with the expression of $B_{\infty,q,n}$, we know that there is a constant $C$ independent of $p_k$ such that
\begin{align*}
\frac{d}{dt}\dint_{\mathbb{R}^n} \rho^{p_k} \,dx
&\leq -3C_1\dint_{\mathbb{R}^n}|\nabla\rho^{\frac{m+p_k-1}{2}}|^2 \,dx \\
&\quad +\frac{2p^2_k}{m}S_n^{-\frac{\theta_1(p_k-m+1)}{m+p_k-1}}\|\nabla c_0\|^2_{L^{\infty}(\R^n)}\|\rho\|^{(1-\theta_1)(p_k-m+1)}_{L^{p_{k-1}}(\R^n)}
\|\nabla\rho^{\frac{m+p_k-1}{2}}\|^{\frac{2\theta_1(p_k-m+1)}{m+p_k-1}}_{L^2(\R^n)}\\
&\quad +\frac{2Cp^2_k}{m}S_n^{-\frac{\theta_1(p_k-m+1)}{m+p_k-1}}\| \rho\|^{2\theta_3+(1-\theta_1)(p_k-m+1)}_{L^{\infty}(0,\infty;L^{p_{k-1}}(\R^n))}
\|\nabla\rho^{\frac{m+p_k-1}{2}}\|^{\frac{2\theta_1(p_k-m+1)}{m+p_k-1}}_{L^2(\R^n)}.
\end{align*}
Young's inequality implies that
\begin{align}
\frac{d}{dt}\dint_{\mathbb{R}^n} \rho^{p_k} \,dx
&\leq -C_1\dint_{\mathbb{R}^n}|\nabla\rho^{\frac{m+p_k-1}{2}}|^2 \,dx\nonumber\\
&\quad +C(C_1)\left(\frac{2p^2_k}{m}S_n^{-\frac{\theta_1(p_k-m+1)}{m+p_k-1}}\|\nabla c_0\|^2_{L^{\infty}(\R^n)}\|\rho\|^{(1-\theta_1)(p_k-m+1)}_{L^{p_{k-1}}(\R^n)}\right)^{q_2}\nonumber\\
&\quad +C(C_1)\left(\frac{2Cp^2_k}{m}S_n^{-\frac{\theta_1(p_k-m+1)}{m+p_k-1}}\| \rho\|^{2\theta_3+(1-\theta_1)(p_k-m+1)}_{L^{\infty}(0,\infty;L^{p_{k-1}}(\R^n))}
\right)^{q_2}.\label{Lp2}
\end{align}
Here $C(C_1)=(C_1 q_1)^{-q_2/q_1}q_2^{-1}$, where $q_1,q_2>1$ satisfy $\frac{q_1\theta_1(p_k-m+1)}{m+p_k-1}=1$ and $\frac{1}{q_1}+\frac{1}{q_2}=1$, which means
\begin{eqnarray}\label{q2}
q_2=\frac{m+p_k-1}{(m+p_k-1)-\theta_1(p_k-m+1)}\leq n.
\end{eqnarray}
It is not difficult to check that there exists a constant $\bar C>1$ independent of $p_k$ such that
\begin{eqnarray*}\label{constant1}
C(C_1)\left(\frac{2}{m}S_n^{-\frac{\theta_1(p_k-m+1)}{m+p_k-1}}\|\nabla c_0\|^2_{L^{\infty}(\R^n)}\right)^{q_2}\leq \bar C,\quad C(C_1)\left(\frac{2C}{m}S_n^{-\frac{\theta_1(p_k-m+1)}{m+p_k-1}}\right)^{q_2}\leq \bar C,
\end{eqnarray*}
from which (\ref{Lp2}) can be reduced to
\begin{align}\label{inequa1}
\frac{d}{dt}\dint_{\mathbb{R}^n} \rho^{p_k} \,dx
&\leq -C_1\dint_{\mathbb{R}^n}|\nabla\rho^{\frac{m+p_k-1}{2}}|^2 \,dx\nonumber\\
 &\quad+\bar Cp^{2q_2}_k\left(\Big(\|\rho\|^{p_{k-1}}_{L^{p_{k-1}}(\R^n)}\Big)^{\eta_2}
+\Big(\|\rho\|^{p_{k-1}}_{L^{\infty}(0,\infty; L^{p_{k-1}}(\R^n))}\Big)^{\eta_3}\right),
\end{align}
where
\begin{eqnarray*}\label{eta1}
\eta_2:=\frac{(p_k-m+1)q_2(1-\theta_1)}{p_{k-1}}\leq 2,\quad \eta_3:=\frac{q_2(2\theta_3+(1-\theta_1)(p_k-m+1))}{p_{k-1}}\leq 2.
\end{eqnarray*}
On the other hand, the interpolation inequality and Sobolev inequality imply that
\begin{eqnarray*}\label{GNS}
\|\rho\|^{p_k}_{L^{p_k}}\leq
S_n^{-\frac{\theta p_k}{m+p_k-1}}
\Big\|\nabla\rho^{\frac{m+p_k-1}{2}}\Big\|^{\frac{2\theta p_k}{m+p_k-1}}_{L^{2}(\R^n)}
\big\|\rho\big\|^{(1-\theta)p_k}_{L^{p_{k-1}}(\R^n)},
\end{eqnarray*}
where
$$
\theta=\frac{\frac{1}{p_{k-1}}-\frac{1}{p_k}}{\frac{1}{p_{k-1}}-\frac{n-2}{n(m+p_k-1)}}\sim O(1),
\quad 1-\theta=\frac{\frac{1}{p_{k}}-\frac{n-2}{n(m+p_k-1)}}{\frac{1}{p_{k-1}}-\frac{n-2}{n(m+p_k-1)}}\sim O(1).
$$
Again using Young's inequality, we have
\begin{eqnarray}\label{young2}
\|\rho\|^{p_k}_{L^{p_k}(\R^n)}\leq (C_1 \ell_1)^{-\ell_2/\ell_1}\ell_2^{-1}  S_n^{-\frac{\ell_2\theta p_k}{m+p_k-1}}
\|\rho\|_{L^{p_{k-1}}(\R^n)}^{p_k((1-\theta))\ell_2}
+C_1\|\nabla \rho^{\frac{m+p_k-1}{2}}\|^2_{L^2(\R^n)},
\end{eqnarray}
where $\ell_1= \frac{m+p_k-1}{\theta p_k}$,
and $\ell_2=\frac{m+p_k-1}{(m+p_k-1)-\theta p_k}$. Notice that there is a constant $\tilde{C}$ independent of $p_k$ such that
$$
(C_1 \ell_1)^{-\ell_2/\ell_1}\ell_2^{-1} S_n^{-\frac{\ell_2\theta p_k}{m+p_k-1}}\leq \tilde C.
$$
Hence from (\ref{inequa1}) and (\ref{young2}), we deduce that
 \begin{align*}
\frac{d}{dt}\|\rho\|^{p_k}_{L^{p_k}(\R^n)}
&\leq -\|\rho\|^{p_k}_{L^{p_k}(\R^n)}
+\tilde C\|\rho\|_{L^{p_{k-1}}(\R^n)}^{p_{k-1}\eta_1}\nonumber\\
&\quad+\bar Cp_k^{2q_2}\left(\Big(\|\rho\|^{p_{k-1}}_{L^{p_{k-1}}(\R^n)}\Big)^{\eta_2}
+\Big(\|\rho\|^{p_{k-1}}_{L^{\infty}(0,\infty;L^{p_{k-1}}(\R^n))}\Big)^{\eta_3}\right).
\end{align*}
where $\eta_1:=\frac{p_k(1-\theta)\ell_2}{p_{k-1}}\leq 2$. In the end we arrive at the following inequality
\begin{align*}
\frac{d}{dt}\|\rho\|^{p_k}_{L^{p_k}(\R^n)} &\leq -\|\rho\|^{p_k}_{L^{p_k}(\R^n)}+C p_k^{2q_2}\Big(\Big(\|\rho\|_{L^{p_{k-1}}(\R^n)}^{{p_{k-1}}}\Big)^{\eta_1}
+\Big(\|\rho\|^{p_{k-1}}_{L^{p_{k-1}}(\R^n)}\Big)^{\eta_2}\nonumber\\
&\quad+\Big(\|\rho\|^{p_{k-1}}_{L^{\infty}(0,\infty;L^{p_{k-1}}(\R^n))}\Big)^{\eta_3}\Big),
\end{align*}
where $0<\eta_1,\eta_2,\eta_3\leq 2$ and $C$ is independent of $p_k$. This together with (\ref{q2}) completes the proof of Lemma \ref{lem2}.
\end{proof}

\bigskip
\noindent
{\bf Proof of Theorem \ref{unithm}}. Let $y_k(t):=\|\rho\|^{p_k}_{L^{p_k}}$, solving the differential inequality (\ref{ODEinequa}), we obtain
\begin{align}\label{inequa4}
\big(e^t y_k(t)\big)'&\leq C p_k^{2q_2}(y_{k-1}^{\eta_1}+y_{k-1}^{\eta_2}+\sup_{t\geq 0}y_{k-1}^{\eta_3})e^t\nonumber\\
&\leq 3 C (4n)^{2n}4^{nk}\max\{1,\sup_{t\geq 0}y_{k-1}^2(t)\}e^t,
\end{align}
Let $a_{k}:=3 C (4n)^{2n}4^{kn}>1$,
$K_0:=\max\{1,m_0,\|\rho_0\|_{L^{\infty}(\R^n)}\}$, $K=K_0^{\frac{p_k}{2^k}}$.
Then we have
\begin{eqnarray}\label{cesti3.26}
y_k(0):=\|\rho_0\|_{L^{p_k}(\R^n)}^{p_k}\leq\Big(\max~\{m_0,\|\rho_0\|_{L^{\infty}(\R^n)}\}\Big)^{p_k}
\leq K_0^{p_k}=K^{2^k}.
\end{eqnarray}
Integrating \eqref{inequa4} from 0 to $t$ and using (\ref{cesti3.26}), we obtain
\begin{align}\label{cesti3.27}
y_k(t)&\leq a_k \max\{1,
\sup_{t\geq0}y_{k-1}^2(t)\}(1-e^{-t})+y_k(0)e^{-t}\nonumber\\
&\leq2a_k \max\{1,
\sup_{t\geq0}y_{k-1}^2(t),y_{k}(0)\}\nonumber\\
&\leq2a_k \max\{
\sup_{t\geq0}y_{k-1}^2(t),K^{2^k}\}.
\end{align}
From \eqref{cesti3.27} after $k-1$ times iteration, we get
\begin{align}\label{cesti3.28}
y_k(t)&\leq(2a_k)(2a_{k-1})^2(2a_{k-2})^{2^2}\cdots(2a_1)^{2^{k-1}}\max\{
\sup_{t\geq0}y_0^{2^k}(t),K^{2^k}\}\nonumber\\
&=\big(6 C(4n)^{2n}\big)^{2^k-1}(4^{n})^{2\cdot2^k-k-2}\max\{\sup_{t\geq0}y_0^{2^k}(t),K^{2^k}\}.
\end{align}
Taking the power $\frac{1}{p_k}$ to \eqref{cesti3.28}, we have
\begin{eqnarray}\label{cesti3.29}
\|\rho\|_{L^{p_k}}\leq6 C(4n)^{2n}4^{2n}\max\{
\sup_{t\geq0}y_0(t),K\}.
\end{eqnarray}

On the other hand, by (\ref{Lp}) we know that $y_0(t)=\|\rho(\cdot,t)\|_{L^{p_0}(\R^n)}^{p_0}$ can be bounded by the following form
\begin{eqnarray}\label{cesti3.24}
y_0(t)=\|\rho\|_{L^{p_0}(\R^n)}^{p_0}\leq C.
\end{eqnarray}
Finally the deserved estimate in (\ref{1.7}) follows from \eqref{cesti3.29} and \eqref{cesti3.24}.


\begin{thebibliography}{001}
\normalsize


\bibitem{BLZ}
S.~Bian, J.-G.~Liu and C.~Zou,
Ultra-contractivity for Keller-Segel model with diffusion exponent $m>1-2/d$,
Kinet. Relat. Models, {\bf 7} (2014), 9--28.

\bibitem{BCL09}
A. Blanchet, J. A. Carrillo and P. Laurencot, Critical mass for a Patlak-Keller-Segel model with
degenerate diffusion in higher dimensions, Calc. Var., {\bf 35} (2009), 133--168.

\bibitem{B13}
 A. Blanchet, P. Lauren\c{c}ot, The Parabolic-Parabolic Keller-Segel System with Critical Diffusion as a Gradient Flow in $\mathbb{R}^d$, $d\geq 3$, Comm. Partial Differential Equations, {\bf 38} (2012), 658--686.

\bibitem{CYL}
X. Chen, A. J\"{u}ngel and J.-G. Liu, A note on Aubin-Lions-Dubinski lemma, Acta Appl. Math., {\bf 133} (2014), 33--43.



\bibitem{CLW1}
L. Chen, J.-G. Liu and J. Wang, Multidimensional degenerate Keller-Segel system with critical diffusion exponent $2n/(n+2)$, SIAM J. Math. Anal., {\bf 44} (2012), 1077--1102.

\bibitem{CW}
L.~Chen and J.~Wang,
Exact criterion for global existence and
blow up to a degenerate Keller-Segel system,
Doc. Math., {\bf 19} (2014), 103--120.


\bibitem{DELL}
J.~Dolbeault, M.J.~Esteban, A.~Laptev and M.~Loss,
One-dimensional Gagliardo-Nirenberg-Sobolev inequalities:
remarks on duality and flows, J. London Math. Soc., {\bf 90} (2014), 525--550.

\bibitem{E}
L. C. Evans, Partial differential equations, Vol. 19 of grad. Stud. Math.,
AMS, Providence, 2002.

\bibitem{IY12}
S. Ishida and T. Yokota, Global existence of weak solutions to quasilinear degenerate Keller-Segel systems of parabolic-parabolic type, J. Differential Equations, {\bf 252} (2012), 1421--1440.


\bibitem{LiebAnalysis}
E.H. Lieb and M. Loss, Analysis, Graduate Studies in Mathematics, 14, {\em American
Mathematical Society Providence, Rhode Island,} 2nd edition. 2001.

\bibitem{L}
P.-L. Lions, Sym\'{e}trie et compacit\'{e} dans les espaces de sobolev, J. Funct. Anal., {\bf 49}  (1982), 315--334.

\bibitem{LW}
J.-G. Liu, J. Wang, A note on $L^\infty$-estimate and uniqueness to a degenerate Keller¨CSegel model, Acta Appl. Math., {\bf 142} (2016), 173--188.

\bibitem{LW1}
J.-G. Liu, J.H. Wang, A generalized Sz. Nagy inequality in higher dimensions and the critical thin film equation, Nonlinearity, {\bf 30} (2017), 35--60.

\bibitem{MJ}
H. Matthias and P. Jan, Heat kernels and maximal $L^p$-$L^q$ estimates for
parabolic evolution equations, Comm. Partial Differential Equations, {\bf 22} (1997), 1647--1669.




\bibitem{S06}
Y. Sugiyama, Global existence in sub-critical cases and finite time blow-up in super-critical cases to degenerate Keller-Segel systems, Diff. Int. Eqns, {\bf 19} (2006), 841--876.


\bibitem{SK06}
Y. Sugiyama and H. Kunii, Global existence and decay properties for a degenerate Keller-Segel model with a power factor in drift term, J. Differential Equations, {\bf 227} (2006), 333--364.

\bibitem{W}
M.I. Weinstein, Nonlinear Schr\"{o}dinger equations and sharp interpolation estimates, Commun.
Math. Phys., {\bf 87} (1983), 567--576.



\end{thebibliography}
\end{document}